\documentclass[12pt,a4paper]{article}
\usepackage[T2A]{fontenc}
\usepackage[cp1251]{inputenc}
\usepackage{cmap}
\usepackage[english]{babel}
\usepackage{amsmath,amsfonts,amssymb, amsthm}
\usepackage{graphicx}
\usepackage{cite}
\usepackage{xcolor}
\usepackage{hyperref}
\usepackage[left=2cm,right=2cm,top=2cm,bottom=2cm,bindingoffset=0cm]{geometry}

\theoremstyle{plain}
\newtheorem{theorem}{Theorem}

\newtheorem{proposition}{Proposition}[section]
\newtheorem{lemma}[proposition]{Lemma}

\theoremstyle{definition}

\newtheorem{example}[proposition]{Example}
\newtheorem{definition}{Definition}

\usepackage[auth-sc]{authblk}

\title{On finite groups factorised by submodular subgroups}
\author{Victor S. Monakhov}
\author{Irina L. Sokhor}
\affil{Department of Mathematics and Technologies of Programming,\\
Francisk Skorina Gomel State University, Belarus}
\affil{victor.monakhov@gmail.com, irina.sokhor@gmail.com}
\date{}
\begin{document}
\maketitle

{\footnotesize %
\abstract{A subgroup $H$ of a finite group $G$ is submodular in $G$ if there is a subgroup chain 
$H=H_0\leq\ldots\leq H_i\leq H_{i+1}\leq \ldots \leq H_n=G$
such that $H_i$ is a modular subgroup of $H_{i+1}$ for every~$i$.
We investigate finite factorised groups with submodular primary (cyclic primary) subgroups in factors.
We indicate a general approach to the description of finite groups
factorised by supersolvable submodular  subgroups.
}

}

\section{Introduction}
All groups in this paper are finite. The formations of all abelian, nilpotent, supersolvable
and all groups with abelian Sylow subgroups are denoted by~$\mathfrak{A}$, $\mathfrak N$,
$\mathfrak U$  and $\mathcal{A}$, respectively. If $\mathfrak{X}$ is a formation,
then $\mathfrak{X}_1$ is the class of all groups from $\mathfrak{X}$ with square-free exponents.

The concept of a modular subgroup came from the lattice theory.
A subgroup $H$ of a group $G$ is modular in $G$ if $H$ is a modular element of the
lattice of subgroups of $G$. R.~Schmidt~\cite{sch94} conducted a detailed analysis of
the behavior of modular subgroups in a group.
The concept of submodularity is an extension of the concept of modularity, and
it is a transitive relation in contrast to modularity.

\begin{definition} \label{dfsmod}
A subgroup $H$ of a group $G$ is submodular in $G$ if there is a subgroup chain 
\begin{equation}\label{eqkfsn}
H=H_0\leq\ldots\leq H_i\leq H_{i+1}\leq \ldots \leq H_n=G
\end{equation}
such that $H_i$ is a modular subgroup of $H_{i+1}$ for every~$i$,
\cite[p.~546]{sch94}.
\end{definition}

Groups with submodular subgroups were studied in~\cite{zim89,vas15,ms23,ms23arx}.
In particular, groups with submodular Sylow subgroups were described in~\cite{zim89,vas15,ms23}.
We denote the class of such groups by~$\mathfrak{Z}$.
In~\cite{ms23}, there was introduced and investigated the class $\mathfrak{C}$
of all groups in which every cyclic primary subgroup is submodular.
Since a subnormal subgroup is submodular, we get $\mathfrak{Z}\subset \mathfrak{C}$,
and $C^3_7\rtimes Q\in \mathfrak{C}\setminus \mathfrak{Z}$, where $Q$ is a non-abelian subgroup
of order $3^3$ and exponent~$3$ that acts irreducibly on an elementary abelian group~$C_7 ^3$
of order~$7^3$,~\cite[example~3]{ms23}.

In this paper, we investigate a group~$G=AB$ in the assumption that subgroups $A$ and~$B$ are submodular in~$G$.
We obtain the criterions of supersolvability of a group~$G$ with supersolvable subgroups $A$ and~$B$,
see Proposition~\ref{pu}. We indicate a general approach to the description of a group $G=AB$
with generalized supersolvable submodular factors $A$ and~$B$.
We prove the following theorem.

\begin{theorem}\label{cabz}
Let $A$ and $B$ be submodular subgroups of a group $G=AB$ and let $A,B\in\mathfrak{Z}$.
Then $G\in\mathfrak{Z}$ when one of the following conditions holds: 

$(1)$~$(|G : AF(G)|, |G : BF(G)|) = 1$;

$(2)$~$(|A/A^{\mathcal{A}_1}|, |B /B^{\mathcal{A}_1}|) = 1$.
\end{theorem}

\begin{theorem}\label{cabc}
Let $A$ and $B$ be submodular subgroups of a group $G=AB$ and let $A,B\in\mathfrak{C}$.
The following statements hold.

$(1)$~If $(|G : AF(G)|, |G : BF(G)|) = 1$, then $G\in\mathfrak{C}$.

$(2)$~If $G^{\mathcal{A}}\in\mathfrak{N}$, then $G\in\mathfrak{Z}$.
\end{theorem}

\section{Preliminaries}\label{spr}

If $X$ is a subgroup (proper subgroup, normal subgroup)  of a group~$Y$,
then we write $X\le Y$ ($X<Y$, $X\lhd Y$, respectively).

We say that a group $G$ has a Sylow tower of supersolvable type if
$$
|G|={p_1^{\alpha_1}p_2^{\alpha_2}\ldots p_n^{\alpha_n}},
\  {p_1<p_2<\ldots < p_n}, \ \alpha_i\in \mathbb N, \ \forall i,
$$
and  $G$ has a normal series $G=G_0\ge G_1\ge \ldots \ge G_{n-1}\ge G_n=1$
such that $G_{i-1}/G_i$ is isomorphic to a Sylow $p_i$-subgroup of $G$ for every~$i$.
The class of all groups with Sylow towers of supersolvable type is denoted by $\mathfrak D$.
It is clear that $\mathfrak{D}\subset\mathfrak{S}$.

We say that a group $G$ is a primitive group if $G$ contains a maximal subgroup $M$ with $M_G=1$.
It is easy to check that the following statement is true.

\begin{lemma} \label{lprimitivef}
Let $\mathfrak{F}$ be a saturated formation and let $G$ be a solvable group.
If $G\notin\mathfrak{F}$ and $G/N\in\mathfrak{F}$ for all non-trivial normal subgroups $N$ of $G$,
then $G$ is a primitive group.
\end{lemma}

\begin{lemma} [{\cite[Theorem~1.1.7,\,1.1.10]{ballcl}}] \label{lprimitive}
Let $G$ be a solvable primitive group and let $M$ be a maximal subgroup of $G$ with $M_G=1$.
The following statements hold.

$(1)$~$\Phi(G) = 1$.

$(2)$~$G$ contains a unique minimal normal subgroup $N$ such that $N = C_G(N)=F(G)=O_p(G)$ for a prime $p$.

$(3)$~$G=F(G)\rtimes M$ and $O_p(M)=1$.
\end{lemma}


The following lemma contains the well-known properties of submodular subgroups.

\begin{lemma} [{\cite[Lemma~1]{zim89}}] \label{lsubm}
Let $G$ be a group, let $H$ and $K$ be subgroups of~$G$ and let $N$ be a
normal subgroup of~$G$.

$(1)$~If $H$ is submodular in $K$ and $K$ is submodular in $G$, then $H$
 is submodular in~$G$.

$(2)$~If $H$ is submodular in $G$, then $H\cap K$ is submodular in~$K$.

$(3)$~If $H/N$ is submodular in $G/N$, then $H$ is submodular in~$G$.

$(4)$~If $H$ is submodular in $G$, then $HN/N$ is submodular in $G/N$
and $HN$ is submodular in~$G$.

$(5)$~If $H$ is subnormal in $G$, then $H$ is submodular in~$G$.
\end{lemma}

Notice that a subgroup $M$ of a group $G$ is a maximal modular subgroup
if $M$ satisfies the following conditions:
$M$ is a proper subgroup of~$G$;  $M$ is a modular subgroup of~$G$;
if $M<K<G$, then $K$ is not modular in~$G$.

\begin{lemma} [{\cite[Lemma~5.1.2]{sch94}}] \label{l512}
If $M$ is a maximal modular subgroup of a group $G$, then either $M$ is normal in $G$ and
$G/M$ is simple or $G/M_G$ is non-abelian of order $pq$ for primes $p$ and~$q$.
\end{lemma}

\begin{lemma}[{\cite[Corollary~2.2]{ms23arx}}] \label{labs}
Let $A$ be a solvable submodular subgroup of a group $G$.
If $B$ is a solvable subgroup of a group $G$ and $G=AB$, then $G$ is solvable.
\end{lemma}

Remind that a subgroup $H$ of a group $G$ is $\mathbb{P}$-subnormal in $G$
if either $H=G$ or there is a subgroup chain~\eqref{eqkfsn} such that $|H_{i+1}:H_i|\in\mathbb{P}$
for all $i$. Here $\mathbb{P}$ is the set of all primes. Properties of $\mathbb{P}$-subnormal subgroups
are detailed in~\cite{vvt2010,mk_rm13,mon_smj16}.

\begin{lemma}\label{lsmp}
If $H$ is a submodular subgroup of a solvable group $G$, then $H$ is $\mathbb{P}$-subnormal in~$G$.
\end{lemma}

\begin{proof}
Let $H$ be a proper submodular subgroup of a solvable group $G$.
In view of Definition~\ref{dfsmod}, there is a maximal modular subgroup $M$ of $G$
such that $H$ is a submodular subgroup of $M$. By induction, $H$ is $\mathbb{P}$-subnormal in $M$.
In view of Lemma~\ref{l512}, $|G:M|\in\mathbb{P}$ and $M$ is $\mathbb{P}$-subnormal in $G$.
Therefore $H$ is $\mathbb{P}$-subnormal in $G$ by~\cite[Lemma~3\,(3)]{mk_rm13}.
\end{proof}

\begin{lemma}\label{lzc}
$(1)$~The classes $\mathfrak{Z}$ and $\mathfrak{C}$ are subgroup-closed saturated formations,
$\mathfrak{Z}\subset\mathfrak{C}\subset \mathfrak{D}$.

$(2)$~$\mathfrak{Z}=\mathfrak{N}\mathcal{A}\cap \mathfrak{C}=\mathfrak{N}\mathcal{A}_1\cap \mathfrak{C}$.
\end{lemma}

\begin{proof}
$(1)$~The statement is true in view
of~\cite[Corollary~3.5,\,Proposition~3.3,\,Lemma~3.2]{ms23}.

$(2)$~Let $G\in\mathfrak{Z}$. In view of Statement~(1), $G\in\mathfrak{C}$.
In addition, $G/F(G)\in\mathcal{A}_1$ by~\cite[Corollary~4.1]{ms23}.
Therefore $G\in \mathfrak{N}\mathcal{A}_1\subseteq\mathfrak{N}\mathcal{A}$.
Thus, $G\in\mathfrak{N}\mathcal{A}_1\cap\mathfrak{C}\subseteq\mathfrak{N}\mathcal{A}\cap\mathfrak{C}$.

Conversely, suppose that $G$ is a group of least order such that
$G\in (\mathfrak{N}\mathcal{A}\cap \mathfrak{C})\setminus\mathfrak{Z}$.
In view of Statement~(1), $G$ has a Sylow tower of supersolvable type.
Since $\mathfrak{N}\mathcal{A}\cap\mathfrak{C}$ and $\mathfrak{Z}$ are
subgroup-closed saturated formations, we deduce that $G$ is a primitive group,
and $G=F(G)\rtimes M$, where $M$ is a maximal subgroup of $G$ with $M_G=1$,
by Lemma~\ref{lprimitivef} and Lemma~\ref{lprimitive}. From $G\in\mathfrak{D}$,
it follows that $F(G)=R$ is a Sylow $r$-subgroup of $G$ for $r=\max\pi(G)$, and $R\in\mathfrak{A}_1$.
Since $G\in\mathfrak{N}\mathcal{A}$, we get $G^\mathcal{A}\leq F(G)$ and
$M\cong G/F(G)\in\mathcal{A}$. Consequently, $G\in\mathcal{A}$. From $G\notin\mathfrak{Z}$
and~\cite[Theorem~1.3]{ms23}, it follows that $G$ contains a $\{p,q\}$-subgroup $K$ such that
$K/\Phi(K)\notin\mathfrak{U}_1$. Let $H$ be a Hall $\{p,q\}$-subgroup of $G$ that contains $K$.
Suppose that $H$ is a proper subgroup of~$G$. In that case, $H\in \mathfrak{Z}$ by induction,
and $K\in\mathfrak{Z}$. Consequently, $K/\Phi(K)\in\mathfrak{U}_1$  in view of~\cite[Theorem~1.3]{ms23},
a contradiction. Hence $G=R\rtimes Q$, where $Q$ is an abelian Sylow subgroup of~$G$
that irreducible acts on $R$. Consequently, $Q$ is a cyclic subgroup by~\cite[Theorem~6.21]{isaacs}.
From $G\in\mathfrak{C}$, it follows that $Q$ is submodular in $G$ and $G\in\mathfrak{Z}$, a contradiction.
Thus, $\mathfrak{Z}=\mathfrak{N}\mathcal{A}\cap\mathfrak{C}$.
\end{proof}

\begin{example}
The Frobenius group $F_5=C_5\rtimes C_4\in \mathcal{A}$
and $F_5^{\mathcal{A}_1}=C_5\rtimes C_2\notin \mathfrak N$.
Hence $F_5\in \mathfrak{N}\mathcal{A}\setminus \mathfrak{N}\mathcal{A}_1$,
and $\mathfrak{N}\mathcal{A}_1\ne \mathfrak{N}\mathcal{A}$ in general.
\end{example}

We also need the following lemma.

\begin{lemma}\label{lgf}
Let $\mathfrak{F}$ be a formation. If $A$ and $N$ are subgroups of a group $G$, $N\lhd G$,
then $(AN/N)^\mathfrak{F}=A^\mathfrak{F}N/N$.
\end{lemma}

\begin{proof}
Since $N$ is a normal subgroup of $G$, we get $H=AN\leq G$.
In view of~\cite[Lemma~2.2.8\,(1)]{ballcl},
$$
(AN/N)^\mathfrak{F}=(H/N)^\mathfrak{F}=H^\mathfrak{F}N/N=A^\mathfrak{F}N/N,
$$
since $A^\mathfrak{F}N=H^\mathfrak{F}N$ by~\cite[Lemma~2.2.8\,(2)]{ballcl}.
\end{proof}

\section{The Main Results}

Groups factorised by $\mathbb{P}$-subnormal subgroups were investigated
in~\cite{vvt2012,montr20ca,montr20,mon2022mz}. The review of these results with proofs
was presented in the monograph of A.\,A.~Trofimuk~\cite{tr21}.

Lemma~\ref{labs} and Lemma~\ref{lsmp} allow to apply known facts about
factorised groups
with $\mathbb{P}$\nobreakdash-\hspace{0pt}subnormal factors to groups factorised by
submodular subgroups.
In particular, we have

\begin{lemma}\label{ld}
Let $A$ and $B$ be submodular subgroups of a group $G$ and let $G=AB$.
If $A$ and $B$ has a Sylow tower of supersolvable type,
then $G$ has a Sylow tower of supersolvable type.
\end{lemma}

\begin{proof}
By Lemma~\ref{labs}, $G$ is solvable. Hence $A$ and~$B$ are $\mathbb{P}$-subnormal in~$G$
in view of Lemma~\ref{lsmp}. By~\cite[Theorem~4.4]{vvt2012}, $G$ has a Sylow tower of supersolvable type.
\end{proof}

Remind that a group $G$ is a siding group if every subgroup of the derived subgroup $G'$ is normal in $G$.
Every siding group is supersolvable. Metacyclic groups, $t$-groups
(solvable groups in which every subnormal subgroup is normal) are siding-groups.
By~$\mathfrak{B}(G)$ we denote the intersection of all normal subgroups $N$ of a group $G$
such that $|\pi(G/N)|\leq 2$.

\begin{proposition}\label{pu}
Let $A$ and $B$ be supersolvable submodular subgroups of a group $G$ and let $G=AB$.
Then~$G$ is supersolvable when one of the following conditions holds: 

$(1)$~$G^\prime$ is nilpotent;

$(2)$~$|G:A|=r^\alpha$, $r\in\pi(G)$, $G$ is $r$-closed;

$(3)$~$|G:A|=r^\alpha$, $r=\max \pi(G)$,

$(4)$ $B$ is nilpotent and normal in $G$;

$(5)$ $B$ is nilpotent and $|G:B|\in\mathbb{P}$;

$(6)$ $B$ is normal in $G$ and a siding group.

Moreover, $G^\mathfrak{U}=G^{\mathfrak{N}^2}\cap\mathfrak{B}(G)$
if either $G^\mathcal{A}\in\mathfrak{N}$ or $(|G:A|,|G:B|)=1$.
\end{proposition}

\begin{proof}
By Lemma~\ref{labs} and Lemma~\ref{lsmp}, $G$ is solvable, $A$ and~$B$ are $\mathbb{P}$-subnormal in~$G$.
In view of~\cite[Corollary~4.7.2]{vvt2012} for Statement~$(1)$ and
\cite[Theorem~3.4,\,Corollary~3.5,\,Theorem~3.7\,(1)--(3)]{montr20} for Statements~$(2)$--$(6)$, respectively,
and also~\cite[Theorem~3.3]{montr20}, we get the conclusion of the proposition.
\end{proof}

Similarly, using theorems of papers~\cite{vvt2012,montr20ca,mon2022mz,montr20},
we can obtain the results for cases when factors~$A$ and~$B$ of a
group~$G=AB$ are submodular
and belong to the formations of all $r$\nobreakdash-\hspace{0pt}supersolvable groups, $\mathrm{w}$-supersolvable groups,
$\mathrm{v}$-supersolvable groups, $\mathrm{sh}$-supersolvable groups.

\begin{lemma}\label{lnilpnormz}
Let $A$ be a submodular subgroup of a group $G$,
let $B$ be a nilpotent normal subgroup, and let $G = AB$.
The following statements hold.

$(1)$~If $A\in\mathfrak{Z}$, then $G\in\mathfrak{Z}$.

$(2)$~If $A\in\mathfrak{C}$, then $G\in\mathfrak{C}$.
\end{lemma}

\begin{proof}[Proof]
$(1)$~Let $G\notin\mathfrak{Z}$. In that case, $G$ contains a Sylow $q$-subgroup $Q$ that is not
submodular in $G$. In view of~\cite[VI.4.6]{hup}, $Q=A_qB_q$, where $A_q$ is a Sylow $q$-subgroup of~$A$
and $B_q$ is a Sylow $q$-subgroup of $B$. Since $A\in\mathfrak{Z}$, we get $A_q$ is submodular in $A$.
Consequently, $A_q$ is submodular in $G$ by Lemma~\ref{lsubm}\,(1). From $B\in\mathfrak{N}$, it follows that
$B_q$ is a characteristic subgroup of $B$. Since $B$ is normal in $G$, we get $B_q$ is normal in $G$.
Consequently, $Q=A_qB_q$ is submodular in $G$ by Lemma~\ref{lsubm}\,(4),
a contradiction.

$(2)$~Assume that $G$ is a counterexample of least order.
By Lemma~\ref{lzc}\,(1), $\mathfrak{C}\subset \mathfrak{D}$,
it follows that $G$ has a Sylow tower of supersolvable type in view of Lemma~\ref{ld}.
By induction, $G/N\in\mathfrak{C}$ for every non-trivial normal subgroup $N$.
Since $\mathfrak{C}$ is a saturated formation by Lemma~\ref{lzc}\,(1),
it follows that $G=F(G)\rtimes M$, where $M$ is a maximal subgroup of~$G$
with $M_G=1$, by Lemma~\ref{lprimitivef} and Lemma~\ref{lprimitive}.
From $G\in\mathfrak{D}$, it follows that $F(G)=R$ is a Sylow $r$-subgroup of $G$ for $r=\max\pi(G)$.
Since $B$ is a nilpotent normal subgroup of $G$, we have $B=F(G)=R$.
Hence without loss of generality we can assume that $A=M$.
By the choice of $G$, there is a cyclic primary subgroup $X$ that is not submodular in $G$.
If $X\leq F(G)$, then $X$ subnormal in $G$. Hence by Lemma~\ref{lsubm}\,(5),
$X$ is submodular in $G$, a contradiction. Therefore without loss of generality we can assume that
$X\leq A$. From $A\in\mathfrak{C}$, it follows that $X$ is submodular in $A$.
Consequently, $X$ is submodular in $G$ by Lemma~\ref{lsubm}\,(1), a contradiction.
\end{proof}

\begin{proof}[Proof of Theorem~{\upshape\ref{cabz}}]
$(1)$~Let $Q$ be a Sylow $q$-subgroup of~$G$ for a prime $q\in\pi(G)$.
Since $(|G : AF(G)|, |G : BF(G)|) = 1$, we deduce that $q$ does not divide $|G : AF(G)|$ or $|G : BF(G)|$.
Therefore  without loss of generality we can assume  that $Q\leq AF(G)$.
According to Lemma~\ref{lsubm}\,(2), $A$ is submodular in $AF(G)$.
Consequently, $AF(G)\in\mathfrak{Z}$  by Lemma~\ref{lnilpnormz}\,(1) and
$Q$ is submodular in $AF(G)$. In view of Lemma~\ref{lsubm}\,(4),
$AF(G)$ is submodular in $G$. Therefore $Q$ is submodular in $G$ by Lemma~\ref{lsubm}\,(1). Thus,
every Sylow subgroup of $G$ is submodular in $G$ and $G\in\mathfrak{Z}$.

$(2)$~Assume that $G$ is a counterexample of least order. Since $\mathfrak{Z}\subset\mathfrak{D}$
by Lemma~\ref{lzc}\,(1), $G$ has a Sylow tower of supersolvable type in view of Lemma~\ref{ld}.
Let $N$ be a non-trivial normal subgroup of~$G$. In that case, $AN/N$ and $BN/N$ are
submodular in $G/N$ by Lemma~\ref{lsubm}\,(4), and
$$
G=(AN/N)(BN/N), \ AN/N\cong A/A\cap N \in\mathfrak{Z}, \
BN/N\cong B/B\cap N \in\mathfrak{Z},
$$
since $\mathfrak{Z}$ is a formation by Lemma~\ref{lzc}\,(1).
Moreover, in view of Lemma~\ref{lgf},
$$
\left(\left|(AN/N)/(AN/N)^{\mathcal{A}_1}\right|, \left|(BN/N) /(BN/N)^{\mathcal{A}_1}\right|\right) =
\left(\left|AN/A^{\mathcal{A}_1}N\right|, \left|BN /B^{\mathcal{A}_1}N\right|\right) =
$$
$$
\left(
|A/A^{\mathcal{A}_1}|/|A\cap N:A^{\mathcal{A}_1}\cap N|,
|B/B^{\mathcal{A}_1}|/|B\cap N:B^{\mathcal{A}_1}\cap N|
\right)=1
$$
Hence $G/N\in\mathfrak{Z}$ by induction.
Consequently, $G$ is a primitive group by Lemma~\ref{lprimitivef},
and $G=F(G)\rtimes M$, where $M$ is a maximal subgroup with~$M_G=1$ by Lemma~\ref{lprimitive}.
From $G\in\mathfrak{D}$, it follows that $F(G)=R$ is a Sylow $r$-subgroup of $G$ for $r=\max \pi(G)$.
In view of Lemma~\ref{lnilpnormz}\,(1), $AR$ and $BR$ belong to $\mathfrak{Z}$.
Therefore $AR$ and $BR$ are proper subgroups of $G$. In addition,
$(AR)^{\mathcal{A}_1}\in\mathfrak{N}$ and $(BR)^{\mathcal{A}_1}\in\mathfrak{N}$
by Lemma~\ref{lzc}\,(2). From $R=C_G(R)$ it follows that
$(AR)^{\mathcal{A}_1}$ and $(BR)^{\mathcal{A}_1}$ are $r$-subgroups.
Since $AR/(AR)^{\mathcal{A}_1}\in\mathcal{A}_1$, we have $A_{r'}\in\mathcal{A}_1$.
A Sylow $r$-subgroup $A_r$ of $A$ is contained in $R\in\mathfrak{A}_1$,
therefore $A\in\mathcal{A}_1$ and $A^{\mathcal{A}_1}=1$. Similarly,
$B\in\mathcal{A}_1$ and $B^{\mathcal{A}_1}=1$. Hence
$(|A|,|B|) = (|A/A^{\mathcal{A}_1}|,|B/B^{\mathcal{A}_1}|)=1$  and $G\in \mathfrak{Z}$
by Statement~(1), a contradiction.
\end{proof}

\medskip

\noindent{\bf Corollary A.1} (\cite[Theorem~3.3\,(1)]{vv16}).
{\sl
Let $A$ and $B$ be submodular subgroups of a group $G=AB$.
If $A,B\in\mathfrak{Z}$ and $(|G:A|,|G:B|)=1$, then $G\in\mathfrak{Z}$.
}

\medskip

\begin{proof}[Proof of Theorem~{\upshape\ref{cabc}}]
$(1)$~Let $X$ be a cyclic $q$-subgroup of $G$ for a prime $q\in\pi(G)$.
Since $(|G : AF(G)|, |G : BF(G)|) = 1$, we deduce that $q$ does not divide $|G : AF(G)|$ or $|G : BF(G)|$.
Consequently, without loss of generality we can assume  that $X\leq AF(G)$. By Lemma~\ref{lsubm}\,(2),
$A$ is submodular in $AF(G)$. Hence $AF(G)\in\mathfrak{C}$ by Lemma~\ref{lnilpnormz}\,(2).
Consequently, $X$ is submodular in $AF(G)$. In view of Lemma~\ref{lsubm}\,(4),
$AF(G)$ is submodular in $G$. Therefore $X$ is submodular in $G$ by Lemma~\ref{lsubm}\,(1).
Thus, every cyclic primary subgroup of $G$ is submodular in $G$ and $G\in\mathfrak{C}$.

$(2)$~Assume that $G$ is a counterexample of least order.
Since $\mathfrak{C}\subset\mathfrak{D}$ by Lemma~\ref{lzc}\,(1), $G$ has a Sylow tower of supersolvable type
in view of Lemma~\ref{ld}. Let $N$ be a non-trivial normal subgroup of $G$. In that case,
$$
(G/N)^{\mathcal{A}} \cong G^{\mathcal{A}}N/N \cong G^{\mathcal{A}}/G^{\mathcal{A}}\cap N\in\mathfrak{N}.
$$
Therefore $G/N\in\mathfrak{Z}$ by induction. According to Lemma~\ref{lprimitivef} and Lemma~\ref{lprimitive},
$G=F(G)\rtimes M$, where $M$ is a maximal subgroup with~$M_G=1$. Since $G\in\mathfrak{D}$,
we get $F(G)=R$  is a Sylow $r$-subgroup of $G$ for $r=\max\pi(G)$ and $R\in \mathfrak{A}_1$.
From $G^{\mathcal{A}}\in\mathfrak{N}$, it follows that $G^{\mathcal{A}}\leq F(G)$,
and $M\cong G/F(G)\in \mathcal{A}$. Thus, $G\in\mathcal{A}$ and $G^{\mathcal{A}}=1$.
By the choice of $G$, $G\notin\mathfrak{Z}$, therefore $G$ contains a $\{p,q\}$-subgroup $K$
such that $K/\Phi(K)\notin\mathfrak{U}_1$ by~\cite[Theorem~1.3]{ms23}. According to~\cite[VI.4.6]{hup},
there are Hall $\{p,q\}$-subgroups $A_{\{p,q\}}$ and $B_{\{p,q\}}$ of $A$ and $B$, respectively,
such that $H=A_{\{p,q\}}B_{\{p,q\}}$ is a Hall $\{p,q\}$-subgroup of~$G$.
Since $A_{\{p,q\}}=A\cap H$, $A_{\{p,q\}}$ is submodular in $H$ by Lemma~\ref{lsubm}\,(2).
Similarly, $B_{\{p,q\}}$ is submodular in $H$. In addition, $A_{\{p,q\}},B_{\{p,q\}}\in\mathfrak{C}$
and $H^{\mathcal{A}}=1$. Suppose that $H$ is a proper subgroup of $G$. In that case, $H\in\mathfrak{Z}$ by induction,
and $K\in\mathfrak{Z}$. Consequently, $K/\Phi(K)\in\mathfrak{U}_1$ by~\cite[Theorem~1.3]{ms23},
a contradiction. Hence, $G=R\rtimes Q$, where $Q$ is an abelian Sylow $q$-subgroup of $G$
that acts irreducibly on~$R$. Consequently, $Q$ is cyclic by~\cite[Theorem~6.21]{isaacs}.
According to~\cite[VI.4.6]{hup}, $Q=A_qB_q$, where $A_q$ and $B_q$ are
Sylow $q$-subgroups of $A$ and $B$, respectively.  Since $Q$ is cyclic, we deduce that
$Q=A_q$ or $Q=B_q$. Therefore  without loss of generality we can assume  that $Q=A_q$.
Since $A\in\mathfrak{C}$, we get $Q$ is submodular in $A$, and $Q$ is submodular in $G$
by Lemma~\ref{lsubm}\,(1), because $A$ is submodular in $G$. Thus, $G\in\mathfrak{Z}$.
\end{proof}

\medskip

\noindent{\bf Corollary B.1} (\cite[Theorem~3.3\,(3)]{vv16}).
{\sl
Let $A$ and $B$ be submodular subgroups of a group $G=AB$.
If $A,B\in\mathfrak{Z}$ and $G^{\mathcal{A}_1}\in \mathfrak N$,
then $G\in\mathfrak{Z}$.
}

\medskip

\noindent{\bf Corollary B.2} (\cite[Proposition~2.6]{vas15}).
{\sl
If $A$ and $B$ are nilpotent submodular subgroups of a group $G=AB$,
then $G\in \mathfrak{Z}\cap \mathfrak U$.
}

\medskip

\begin{proof}
By Lemma~\ref{labs}, $G$  is solvable. Therefore $A$ and $B$ are $\mathbb{P}$-subnormal in~$G$ by Lemma~\ref{lsmp}.
Now, $G\in \mathfrak U$ by~\cite[Theorem~4.2.1\,(1)]{tr21} and $G^\prime$ is nilpotent.
Since every subgroup of a nilpotent group is subnormal,
we have $A,B\in \mathfrak C$ in view of Lemma~\ref{lsubm}\,(6).
From $G^{\mathcal A}\le G^\prime$, it follows that $G\in \mathfrak Z$ by Theorem~\ref{cabc}\,(2).
\end{proof}

\begin{example}
In the Frobenius group $F_5=C_5\rtimes C_4$,  $A=C_5$ and $B=C_4$ are
$\mathbb{P}$\nobreakdash-subnormal, belong to $\mathfrak{Z}\subset\mathfrak{C}$ and
have coprime indices in $G$. Since $B_{F_5}=1$ and $F_5\notin\mathfrak{U}_1$,
$B$ is not submodular in $F_5$. Therefore, $F_5\in \mathcal{A}\setminus\mathfrak{Z}$.
Consequently, in Theorem~\ref{cabz} and Theorem~\ref{cabc},
the condition of submodularity of  at least one factor could not be weakened
to $\mathbb{P}$-subnormality.
\end{example}

\end{document}